\documentclass[reqno,12pt]{amsart}

\usepackage{amsmath, amsthm, amsfonts, amssymb,color}
\input{mathrsfs.sty}

\usepackage{tikz}
\usetikzlibrary{arrows.meta}

\usepackage{amsmath}
\usepackage{amsfonts}
\usepackage{amssymb}
\usepackage{eufrak}
\usepackage{amscd}
\usepackage{amsthm}
\usepackage{amstext}
\usepackage[all]{xy}

   \theoremstyle{plain}
   \newtheorem{thm}{Theorem}
   
   \newtheorem{lem}[thm]{Lemma}
   \newtheorem{cor}[thm]{Corollary}
   \theoremstyle{definition}

   \theoremstyle{remark}

\author{V. Manuilov}

\date{}

\address{Moscow Center for Fundamental and Applied Mathematics, Moscow State University,
Leninskie Gory 1, Moscow, 
119991, Russia}

\email{manuilov@mech.math.msu.su}

\title {Inverse semigroup of metrics on a double is fundamental}

\begin{document}

\begin{abstract}
We show that the inverse group of equivalence classes of metrics on two copies of a metric space is fundamental.

\end{abstract}

\maketitle

In \cite{M} we have shown that the set $M(X)$ of equivalence classes of metrics on two copies of a metric space $X$ has a rigid algebraic structure of an inverse semigroup. This, already, was unexpected. Now we go even further and show that this inverse semigroup is fundamental for any metric space $X$. Recall that a semigroup $S$ is an inverse semigroup if for any $s\in S$ there exists a unique pseudoinverse $t\in S$ ($t$ is pseudoinverse for $s$ if $sts=s$ and $tst=t$). We denote the pseudoinverse for $s$ by $s^*$. An inverse semigroup $S$ acts on the set $E=E(S)$ of its idempotents by $\mu_s(e)=ses^*$, $e\in E$. $S$ is \emph{fundamental} if $\mu_s\neq\mu_t$ for $s\neq t$. This allows to faithfully represent $S$ as a subset of self-maps of $E$. Our standard reference to the inverse semigroup theory is \cite{Lawson}.

Let $X=(X,d_X)$ be a metric space. 
A {\it double} of $X$ is a metric space $X\times\{0,1\}$ with a metric $d$ such that 
the restriction of $d$ on each copy of $X$ in
$X\times\{0,1\}$ equals $d_X$, and 
the distance between the two copies of $X$ is non-zero.
Let $\mathcal M(X)$ denote the set of all such metrics.

We identify $X$ with $X\times\{0\}$, and write $X'$ for $X\times\{1\}$. Similarly, we write 
$x$ for $(x,0)$ and $x'$ for $(x,1)$, $x\in X$. 
Note that metrics on a double of $X$ may differ only when the two points lie in different copies of $X$, so
to define a metric $d$ in $\mathcal M(X)$ it suffices to define $d(x,y')$ for all $x,y\in X$. The assumption that $d(X,X')>0$ means that there exists $c>0$ such that $d(x,y')\geq c$ for any $x,y\in X$.

Recall that two metrics, $d_1$, $d_2$, on the double of $X$ are \emph{coarsely equivalent}  if there exists a homeomorphism $\phi$ of $[0,\infty)$ such that 
$$
\phi^{-1}(d_1(x,y'))\leq d_2(x,y')\leq \phi(d_1(x,y'))
$$
for any $x,y\in X$. 
In this case we write $d_1\sim d_2$, or $[d_1]=[d_2]$. The quotient space $\mathcal M(X)/\sim$ we denote by $M(X)$.

The formula
\begin{equation}\label{composition}
(d_1 d_2)(x,y')=\inf_{u\in X}[d_1(x,u')+d_2(u,y')],\quad x\in X,z\in Z,
\end{equation} 
defines a metric $d_1 d_2$, $[d_1 d_2]\in\mathcal M(X)$, and it was shown in \cite{M} that $M(X)$ with this multiplication operation is an inverse semigroup with $\mathbf 0$ and $\mathbf 1$.

Recall that there is a partial order $\preceq$ on $M(X)$ from the inverse semigroup structure. For $s,t\in M(X)$, $s\preceq t$ if $s=et$ for some idempotent element $e\in M(X)$. In geometric terms, if $d,b\in\mathcal M(X)$, $s=[d]$, $t=[b]$ then $s\preceq t$ if there exists a homeomorphism $\phi$ of $[0,\infty)$ such that $b(x,y')\leq\phi(d(x,y'))$ for any $x,y\in X$. We write $s\prec t$ if $s\preceq t$ and $s\neq t$.

A similar partial order can be defined on the algebra $C_+(Z)$ of continuous $[0,\infty)$-valued functions on a metric space $Z$. For $f,g\in C_+(Z)$ we say that $f\preceq g$ if there exists a homeomorphism $\phi$ of $[0,\infty)$ such that $f(z)\geq \phi(g(z))$ for any $z\in Z$ (note that the order for functions is the opposite to the usual one).  

\begin{lem}\label{lem1}
If $f\preceq g$ does not hold then there exists a sequence $\{z_n\}_{n\in\mathbb N}$ in $Z$ and $C>0$ such that $\lim_{n\to\infty}g(z_n)=\infty$ and $\sup_{n\in\mathbb N}f(z_n)<C$.  

\end{lem}
\begin{proof}
We shall prove the contrapositive: if for any sequence $\{z_n\}_{n\in\mathbb N}$ such that $\lim_{n\to\infty}g(z_n)=\infty$ one has $\lim_{n\to\infty}f(z_n)=\infty$ then $f\geq g$. 

Set $Z_n=\{z\in Z:n-1\leq g(z)\leq n\}$. Then $Z=\cup_{n\in\mathbb N}Z_n$. Let $k_n=\inf_{z\in Z_n}f(z)$. Suppose that the sequence $\{k_n\}$ contains an infinite subsequence $\{k_{n_i}\}_{i\in\mathbb N}$ bounded by some $D>0$. Then for each $i\in\mathbb N$ there exists $z_{n_i}\in Z_{n_i}$ such that $f(z_{n_i})<D+1$. As $g(z_{n_i})>n_i-1$, we obtain a contradiction. Therefore $\lim_{n\to\infty}k_n=\infty$. Let $\phi$ be a homeomorphism of $[0,\infty)$ such that $\phi(n)<k_n$ for each $n\in\mathbb N$. Then, for $z\in Z_n$, one has $\phi(g(z))\leq \phi(n)<k_n\leq f(z)$, hence $f\geq g$.
\end{proof}

\begin{cor}\label{C}
Let $b,c\in\mathcal M(X)$, and let $c\preceq b$ does not hold. Then there exist unbounded sequences $\{x_n\}_{n\in\mathbb N}$ and $\{y_n\}_{n\in\mathbb N}$ in $X$ and $C>0$ such that $\lim_{n\to\infty}c(x_n,y'_n)=\infty$ and $\sup_{n\in\mathbb N}b(x_n,y'_n)<C$. 

\end{cor}
\begin{proof}
Take $Z=X\times X$. Lemma \ref{lem1} provides sequences $\{x_n\}_{n\in\mathbb N}$ and $\{y_n\}_{n\in\mathbb N}$ in $X$ and $C>0$ such that $\lim_{n\to\infty}c(x_n,y'_n)=\infty$ and $\sup_{n\in\mathbb N}b(x_n,y'_n)<C$. It remains to show that they are not bounded. Suppose that both sequences are bounded, i.e. there exists $R>0$ such that $d_X(x_0,x_n)<R$ and $d_X(x_0,y_n)<R$ for some $x_0\in X$ and for any $n\in\mathbb N$. Then 
$$
c(x_n,y'_n)\leq d_X(x_n,x_0)+c(x_0,x'_0)+d_X(x_0,y_n)
$$ 
is bounded --- a contradiction. Thus, at least one of the two sequences, say, $\{x_n\}_{n\in\mathbb N}$, is unbounded. Suppose that the other is bounded. Then 
$$
d_X(x_0,y_n)\geq d_X(x_0,x_n)-b(x_n,y'_n)\geq d_X(x_0,x_n)-C
$$ 
--- a contradiction.
\end{proof}

Given $b,c\in\mathcal M(X)$ such that $c\preceq b$ does not hold, and $\{x_n\}_{n\in\mathbb N}$, $\{y_n\}_{n\in\mathbb N}$ and $C$ as in Corollary \ref{C}, define a metric $a\in\mathcal M(X)$ by 
\begin{equation}\label{a}
a(x,y')=\inf_{n\in\mathbb N}(d_X(x,x_n)+d_X(y_n,y)+C), \quad x,y\in X.
\end{equation} 

\begin{lem}\label{main}
One has $aba^*\prec aca^*$.

\end{lem}
\begin{proof}
By definition ((\ref{composition}) and (\ref{a})), 
\begin{eqnarray*}
(aca^*)(x,y')&=&\inf_{u,v\in X}[a(x,u')+1+c(u,v')+1+a^*(v,y')]\\
&=&\inf_{u,v\in X}[a(x,u')+a(y,v')+c(u,v')+2]\\
&=&\inf_{\substack{u,v\in X,\\n,m\in\mathbb N}}[d_X(x,x_n)+C+d_X(y_n,u)+d_X(y,x_m)+C+d_X(y_m,v)+c(u,v')+2].
\end{eqnarray*}
Similarly, 
$$
(aba^*)(x,y')=\inf_{\substack{u,v\in X,\\n,m\in\mathbb N}}[d_X(x,x_n)+d_X(y_n,u)+d_X(y,x_m)+d_X(y_m,v)+b(u,v')+2C+2].
$$
Fix $n$ and $m$ such that the infimum for $(aca^*)(x,y')$ is attained. Then we have 
$$
(aca^*)(x,y')=\inf_{u,v\in X}[d_X(x,x_n)+d_X(y_n,u)+d_X(y,x_m)+d_X(y_m,v)+c(u,v')+2C+2];
$$
and
\begin{eqnarray*}
(aba^*)(x,y')&\leq&\inf_{u,v\in X}[d_X(x,x_n)+d_X(y_n,u)+d_X(y,x_m)+d_X(y_m,v)+b(u,v')+2C+2]\\
&\leq& [d_X(x,x_n)+d_X(y_n,y_n)+d_X(y,x_m)+d_X(y_m,y_m)+b(x_n,y'_m)+2C+2]\\
&=&d_X(x,x_n)+d_X(y,x_m)+b(x_n,y'_m)+2C+2\\
&\leq&d_X(x,x_n)+d_X(y,x_m)+d_X(x_n,x'_m)+b(x_m,y'_m)+2C+2\\
&\leq&d_X(x,x_n)+d_X(y,x_m)+d_X(x_n,x'_m)+3C+2.
\end{eqnarray*}

Then
\begin{eqnarray*}
(aca^*)(x,y')&=&\inf_{u,v\in X}[d_X(x,x_n)+d_X(y_n,u)+d_X(y,x_m)+d_X(y_m,v)+c(u,v')+2C+2]\\
&=&d_X(x,x_n)+d_X(y,x_m)+3C+2+\inf_{u,v\in X}[d_X(y_n,u)+d_X(y_m,v)+c(u,v')]-C\\
&\geq&(aba^*)(x,y')+\inf_{u,v\in X}[d_X(y_n,u)+d_X(y_m,v)+c(u,v')]-C. 
\end{eqnarray*}

It follows from the triangle inequality that
$$
d_X(x_n,u)+d_X(y_m,v)+c(u,v')\geq c(x_n,y'_m),
$$
hence
\begin{eqnarray*}
(aca^*)(x,y')&\geq&(aba^*)(x,y')+\inf_{u,v\in X}[d_X(y_n,u)+d_X(y_m,v)+c(u,v')]-C\\
&\geq&(aba^*)(x,y')+c(x_n,y'_m)-C.
\end{eqnarray*}

As both $\{x_n\}_{n\in\mathbb N}$ and $\{y_n\}_{n\in\mathbb N}$ are unbounded and $c(x_n,y'_n)\to\infty$, passing to a subsequence, we may assume that for all $n,m\in\mathbb N$ we have $c(x_n,y'_m)>C$, hence for any $x,y\in X$ we have $(aca^*)(x,y')\geq (aba^*)(x,y')$, so $aca^*\preceq aba^*$.

It remains to show that $aba^*\neq aca^*$. 
 We have
\begin{eqnarray*}
(aba^*)(x_k,x'_k)&\leq&\inf_{\substack{u,v\in X,\\n,m\in\mathbb N}}[d_X(x_k,x_n)+d_X(y_n,u)+d_X(x_k,x_m)+d_X(y_m,v)+b(u,v')+2C+2]\\
&\leq&\inf_{n,m\in\mathbb N} [d_X(x_k,x_n)+d_X(x_k,x_m)+b(x_n,y'_m)+2C+2]\\
&\leq&d_X(x_k,x_k)+d_X(x_k,x_k)+b(x_k,y'_k)+2C+2
\ \leq \ 3C+2
\end{eqnarray*}
for any $k\in\mathbb N$, hence 
\begin{equation}\label{f1}
\sup_{k\i\mathbb N}(aba^*)(x_k,x'_k)<\infty. 
\end{equation}

On the other hand, 
\begin{eqnarray*}
(aca^*)(x_k,x'_k)&=&\inf_{\substack{u,v\in X,\\n,m\in\mathbb N}}[d_X(x,x_n)+C+d_X(y_n,u)+d_X(x_k,x_m)+C+d_X(y_m,v)+c(u,v')+2]\\
&\geq&\inf_{n,m\in\mathbb N}[d_X(x_k,x_n)+d_X(x_k,x_m)+c(x_n,y'_m)+2C+2].
\end{eqnarray*}
As $\{x_n\}_{n\in\mathbb N}$ is unbounded, we may pass to a subsequence with the property that $d_X(x_k,\{x_1,\ldots,x_{k-1}\})>2^k$. 
If $n\neq k$ or $m\neq k$ then 
$$
d_X(x_k,x_n)+d_X(x_k,x_m)+c(x_n,y'_m)>2^k. 
$$
If $n=m=k$ then 
$$
d_X(x_k,x_n)+d_X(x_k,x_m)+c(x_n,y'_m)>c(x_k,y'_k). 
$$
In both cases, 
\begin{equation}\label{f2}
\lim_{k\to\infty}(aca^*(x_k,x'_k))=\infty. 
\end{equation}
It follows from (\ref{f1}) and (\ref{f2}) that the metrics $aba^*$ and $aca^*$ are not equivalent. 
\end{proof}

\begin{thm}
The inverse semigroup $M(X)$ is fundamental.

\end{thm}
\begin{proof}
Let $s,t\in M(X)$, $s\neq t$. Then we cannot have both $s\preceq t$ and $t\preceq s$, so at least one of these does not hold. Suppose that $s\preceq t$ is not true. Then, by Lemma \ref{main}, there exists $a\in M(X)$ such that $ata^*\prec asa^*$. 

Note that for an inverse semigroup $S$, $x\prec y$ implies that $xx^*\prec yy^*$. Indeed, $x\prec y$ implies that there exists $f\in E(S)$ such that $x=fy$. Then $xx^*=fyy^*f=fyy^*$, hence $xx^*\preceq yy^*$ (recall that idempotents commute in inverse semigroups). It remains to show that $xx^*\neq yy^*$. Suppose the contrary. Then $fyy^*=xx^*=yy^*$. Multiplying by $y$ from the right, we get $fy=fyy^*y=yy^*y=y$ which implies that $x=y$.

Taking $x=ata^*$, $y=asa^*$, we get $ata^*at^*a\prec asa^*as^*a$. Set $e=a^*a\in E(M(X))$. Then $atet^*a^*\neq ases^*a^*$. But this contradicts $tet^*=ses^*$, hence $tet^*\neq ses^*$, i.e. $\mu_s(e)\neq \mu_t(e)$.   
\end{proof}


\end{document}